\documentclass[11pt,a4paper]{amsart}
\usepackage[latin1]{inputenc}
\usepackage{amsmath}
\usepackage{amsthm}
\usepackage{amsfonts}
\usepackage{amssymb}
\usepackage{graphicx}

\usepackage{hyperref,url}
\usepackage{xcolor}
\hypersetup{
   colorlinks,
    linkcolor={black!50!black},
    citecolor={black!50!black},
    urlcolor={black!80!black}
}
\newtheorem{prop}{Proposition}
\newtheorem{defn}[prop]{Definition}
\newtheorem{lemma}[prop]{Lemma}

\newtheorem{corollary}[prop]{Corollary}

\newcommand{\N}{\mathbb{N}}
\newcommand{\Z}{\mathbb{Z}}
\newcommand{\R}{\mathbb{R}}
\newcommand{\C}{\mathbb{C}}

\newcommand{\ckj}{c_k(j)}
\newcommand{\ckjgen}{c_k^{(s)}(j)}
\newcommand{\mukd}{\mu(\frac{k}{d})}
\newcommand{\kbyds}{\frac{k}{d^s}}
\newcommand{\ghi}{\frac{J_s(k)\mu(k/(j, k^s)_s)}{J_s(k/(j,k^s)_s)}}
\newcommand{\muksds}{\mu(\frac{k^s}{d^s})}
\newcommand{\jsksds}{J_s(\frac{k^s}{d^s})}
\newcommand{\ksds}{\frac{k^s}{d^s}}
\newcommand{\toth}{T\'{o}th}

\begin{document}
\title{Certain weighted averages of generalized Ramanujan sums}
\author{K Vishnu Namboothiri}
\thanks{Baby John Memorial Government College, Chavara, Kollam, Kerala, India\\Affiliated to the University of Kerala, Thiruvananthapuram\\Department of Collegiate Education, Government of Kerala\\
Email : \texttt{kvnamboothiri@gmail.com}\\
The author acknowledges support from the University Grants Commission of India under its Minor Research Project Grant scheme (XII Plan) .}
\maketitle
\begin{abstract}
We derive certain identities involving various known arithmetical functions and a generalized version of Ramanujan sum. L. \toth\ constructed certain weighted averages of Ramanujan sums with various arithmetic functions as weights.  We choose a generalization of Ramanujan sum given by E. Cohen and derive the weighted averages corresponding to the versions of the weighted averages established by \toth.

\end{abstract}

\section{Introduction}
Ramanujan sum is a certain sum of powers of roots of unity defined and used seriously for the first time in the literature by Srinivasa Ramanujan. His paper \emph{On certain trigonometric sums} published in 1918 (\cite{ramanujan}) discusses various infinite series expressions for many well known arithmetical functions in the form of a power series. His interest in the sums originated in his desire to ``obtain expressions for a variety of well-known arithmetical functions of $n$ in the form of a series $\sum_s a_sc_s(n)$''. There have been a lot of analysis dealing with this particular aspect of Ramanujan sums (See, for example, \cite{lucht}).  On the other hand, it found immense use in in classical character theory. It was proved that the sums can be used to establish the integrality of the character values for the symmetric group (\cite{james}). However,  perhaps the most famous appearance of Ramanujan sums is their crucial role in Vinogradov's proof (See \cite{nathan}) that every sufficiently large odd number is the sum of three primes.

In more recent years, Ramanujan sums have appeared in various other seemingly unrelated problems. Being interested in the applications of the Ramanujan sums, many mathematicians later tried to generalize it to find more and more applications. One of the most popular generalization was given by E. Cohen. After that, various other generalizations were discussed in many papers including that of C. S. Venkataraman and R. Sivaramakrisnan \cite{venkata} and  J. Chidambaraswamy (\cite{chid}).

In this paper we are interestd in establishing certain weighted averages of generalized Ramanujan sums. The generalization we choose was given by E. Cohen. We remark that our identities are based on similar kind of averages derived by  L. \toth\ in \cite{toth} using the classical Ramanujan sums. In between, we derive another identity using the Jordan totient function which could be of independent interest.

\section{Notations}

Most of the notations, functions, and identities we introduce below are standard and can be found in \cite{apostol1} or \cite{cohen4}. We use the following standard notations:  $\N$ the set of all natural numbers, $\Z$ the set of all integers, $\R$ the set of all real numbers and $\C$ the set of all complex numbers. $\mu$ will denote the Mobius function on $\N$. Note that for any given $k\in\N$,  we have $$\sum_{d|k}\mu(d) =0 $$
Let $\lfloor x\rfloor$ denote the integer part of $x$, and   $\tau(n)$ and $\sigma(n)$ denote the number and sum of divisors
of $n$ respectively.
For two arithmetical functions $f$ and $g$, $f*g$ denote their Dirichlet
convolution (Dirichlet product). Note that this product is commutative. 
The arithmetic function $N$ is defined as, for $n\in\N$, $$N^\alpha(n) := n^\alpha$$ for any $\alpha\in\R$. For $\alpha=1$, we will simply write $N(n)$ instead of writing $N^1(n)$, and $N^0(n) :=1$ for all $n$.

The von Mangoldt function over $\N$ is denoted by $\Lambda(n)$. The Euler totient function is denoted by $\varphi$ and it satisfies the identity \begin{equation*}
\varphi(n) = n\prod_{\substack{p|n\\p\text{ prime}}}\left(1-\frac{1}{p}\right)
\end{equation*} 

A well known generalization of this identity gives the Jordan totient function of order $s\in\N$:
\begin{equation*}
J_s(n) := n^s\prod_{\substack{p|n\\p\text{ prime}}}\left(1-\frac{1}{p^s}\right)
\end{equation*} 
where $n\in\N$.

We would be using the following identity for the Jordan totient function:
\begin{prop}(\cite{apostol1}, Chapter 2, Exer. 17)
For $n\geq 1$, we have $$J_s(n) = \sum_{d|n} d^s\mu(n/d) = \sum_{d|n} (n/d)^s\mu(d)$$
\end{prop}

For $x\in\R, x>0$, the Gamma function $\Gamma$ is defined by 
\begin{equation*}
\Gamma(x) = \int_{0}^{\infty}e^{-t}t^{x-1} dt
\end{equation*}

Now we recall the definition of the Bernoulli Polynomials and Bernoulli numbers. 
For any $x\in\C$ define the functions $B_n(x)$ by the equation 
\begin{equation*}
\frac{ze^{xz}}{e^z-1} = \sum_{n=0}^{\infty}\frac{B_n(x)}{n!}z^n\text{ where } |z|<2 \pi
\end{equation*}
If we put $x=0$, we get (with the convention that $B_n :=B_n(0)$), then
\begin{equation*}
\frac{z}{e^z-1} = \sum_{n=0}^{\infty}\frac{B_n}{n!}z^n\text{ where } |z|<2 \pi
\end{equation*}
$B_n(x)$ are called Bernoulli polynomials and $B_n=B_n(0)$ are called Bernoulli numbers.

Following are some important properties of Bernoulli polynomials (See section 12.2 of \cite{apostol1}) which we use later:
\begin{prop}
For $n\in\N, n\geq 2$, the Bernoulli numbers satisfy
\begin{enumerate}
\item $B_n(0)=B_n(1)$, and 
\item $B_n = \sum_{k = 0}^n \binom{n}{k}B_k$
\end{enumerate}
\end{prop}

Next we define the generalized gcd function: 
\begin{defn}
The generalized GCD function on $\N\times\N$  denoted by $(j, k^s)_s$ is defined to  give the largest $d\in\N$ such that $d|k$ and $d^s | j$. Therefore,  $(j, k^s)_s = 1$ means that no natural number $d\geq 1$ exists such that $d|k$ and $d^s|j$.
\end{defn}
When $s=1$, the generalized GCD function becomes the usual GCD function.

For $k\in\N,  j\in\Z$, the Ramanujan sum defined as the sum
\begin{equation*}
c_k(j)=\sum_{\substack{m=1\\gcd(m,k)=1}}^{k}e^{2\pi ijm/k}
\end{equation*}

In \cite{cohen1}, E. Cohen generalized this to
\begin{defn}\label{defn:genrs1}
For $s, k\in\N, j\in \Z$, the generalized Ramanujan sum is defined as \begin{equation*}
c_k^{(s)}(j) = \sum_{\substack{m=1\\(j,k^s)_s=1}}^{k}e^{2\pi ijm/k^s}
\end{equation*}
\end{defn}

Recall that, for Ramanujan sum, we have $$c_k(j) = \sum_{d|(k,j)} d\mu(\frac{k}{d}) = \sum_{d|(k,j)} \frac{k}{d}\mu(d)$$ The following similar identity holds for the generalized Ramanujan sum (\cite{cohen1} Section 2):
\begin{prop}
\begin{equation}\label{eq:rs-gen-mu}
c^{(s)}_k(j) = \sum_{d|(j,k^s)_s} d^s\mu(\frac{k}{d}) = \sum_{\substack{d|j^s\\d|k}} d^s\mu(\frac{k}{d}) 
\end{equation}
\end{prop} 
Holder's (See \cite{holder}) evaluation of  the Ramanujan sums is given by 
\begin{equation*}
c_k(j)= \frac{\varphi(k)\mu(k/gcd(k,j))}{\varphi(k/gcd(k,j))}\text{ where }k\in \N, j\in \Z
\end{equation*}
For the generalized sum,  Cohen showed that(\cite{cohen3}, section 2)
\begin{prop}{(Generalized Holder's identity)}
\begin{equation}\label{ghi}
{c_k}^{(s)}(j)= \frac{J_s(k)\mu(k/(j, k^s)_s)}{J_s(k/(j,k^s)_s)} \text{ where } j, k,s\in \N
\end{equation}
\end{prop}
\begin{defn}
For $k, n\in \N$, the function $\theta_k(n)$ is defined as $$\theta_k(n) := \begin{cases}
1, \text {if } (k,n)=1\\
0, \text {if } (k,n) >1
\end{cases}$$
\end{defn}

Let $\mathcal{F}$ denote the $\C-$linear space of all arithmetic functions with the usual operations.
\begin{defn}
An arithmetic function $f\in\mathcal{F}$ is said to be $k-$periodic for a fixed $k\in\N$ if $f(n+k) = f(k)$ for all $n\in\N$.
\end{defn}
The following proposition which we use later, can be found, for example, in \cite{schwarz}.
\begin{prop} 
For $k\in\N$, every $k-$periodic $f\in\mathcal{F}$ has a Fourier expansion of the form $$f(k) = \sum_{j=1}^{k} g(j) exp(2\pi i j n/k)$$ where the Fourier coefficients $g(j)$ are unique, and are infact given by the equation $$g(m) =\frac{1}{k}\sum_{j=1}^{k} f(k) exp(-2\pi i j n/k)$$
\end{prop}
Any other notations and definitions we use will be introduced later as and when required.

\section{Problems}

For $r\in\N$, E. Alkan established an identity for the weighted average $$\frac{1}{k^{r+1}}\sum_{j=1}^kj^rc_k(j)$$ based on which he could establish many results in \cite{alkan1} and \cite{alkan2}. If we denote the above weighted average by $ S_r(k)$,  the identity given by E. Alkan is precisely the following:

\begin{prop}\label{prop:toth-main}
\begin{equation}\label{srk}
 S_r(k)=\frac{\varphi(k)}{2k}+\frac{1}{r+1}\sum_{m=0}^{\lfloor
r/2\rfloor}\binom{r+1}{2m}B_{2m}\prod_{p|k}(1-\frac{1}{p^{2m}})
\end{equation}

\end{prop}
This identity appeared as Eq. 2.19 in \cite{alkan2}
. He used this identity to prove exact formulas for
certain mean square averages of special values of $L-$functions. Alkan described
these mean square averages in \cite{alkan1} and proved an asymptotic formula for $
\sum\limits_{k\leq x}S_r{k}$ in \cite{alkan2} based on the identity given above. How did he proceed to establish the identity is, that, he used Holder's evaluation of  the Ramanujan sums  \cite{holder}  and 
applied the formula 
\begin{equation}
 \sum\limits_{\substack{j=1\\ gcd(j,n)=1}}^{n}
j^r=\frac{n^{r+1}}{r+1}\sum\limits_{m=0}^{\lfloor
r/2\rfloor}\binom{r+1}{2m}\frac{B_{2m}}{n^{2m}}\prod_{p|n}(1-p^{2m-1}) \qquad
(n, r\in \N, n>1)
\end{equation}
(See \cite{singh} for details) and then considered the cases $r$ even and $r$
odd seperately. The identity (\ref{srk}) and its proof appears in \cite{alkan3}
also.

The proof of the identity (\ref{srk}) given in \cite{alkan1} was a little bit complicated. A simpler proof this identity (\ref{srk}) appeared in a paper \cite{toth} of L.\toth. In addition, in the same paper, \toth\ derived identities for some other weighted averages of the Ramanujan sums with weights concerning logarithms, values of arithmetic functions for gcd's, the Gamma function, the Bernoulli numbers, binomial coefficients, and the exponential function. 
Our focus will be on establishing the weighted averages of \toth\ for the generalized Ramanujan sums (of E. Cohen) replacing the classical Ramanujan sums. Considering that the genralized Ramanujan sum with $s=1$ becomes the classical Ramanujan sum, it may not be very surprising to note that, the identities we establish here will resemble \toth's identities when we put $s=1$. In fact, the way in which we proceed to establish the identities closely resembles the steps given in \cite{toth}. 

\section{Alkan's Identity}
As we mentioned earlier, the most important computations in \cite{alkan1} and \cite{alkan2} of E. Alkan depends heavily on the identity given in Proposition (\ref{prop:toth-main}). The identity in its final form consists of the Euler totient function and the Jordan totient function. We propose and prove a similar identity for the generalized Ramanujan sum consisting of similar type of functions. Our identity is precisely the following:
\begin{prop}
For $s, k, r\in\N$, define the weighted average $S_r^{(s)}(k)$ by 
\begin{equation}
S_r^{(s)}(k) = \frac{1}{k^{s(r+1)}}\sum_{j=1}^{k^s}j^{r}c^{(s)}_k(j)
\end{equation}
Then, 
\begin{equation}
S_r^{(s)}(k)=\frac{J_s(k)}{2k}+\frac{1}{r+1}\sum_{m=0}^{\lfloor
r/2\rfloor}\binom{r+1}{2m}B_{2m}\frac{J_s(k)}{k^{2ms}}
\end{equation}
\end{prop}

\begin{proof}
Using identity (\ref{eq:rs-gen-mu}), we get
\begin{eqnarray*}
S_r^{(s)}(k) &=& \frac{1}{k^{s(r+1)}}\sum_{j=1}^{k^s}j^{r}\sum_{\substack{d|j^s\\d|k}} d^s\mu(\frac{k}{d}) \\
&=& \frac{1}{k^{s(r+1)}}\sum_{d|k} d^s\mu(\frac{k}{d})\sum_{n=1}^{k^s/d^s}(nd^s)^r \\
&=& \sum_{d|k} \frac{d^{s(r+1)}}{k^{s(r+1)}} \mu(\frac{k}{d})\sum_{n=1}^{k^s/d^s}n^r
\end{eqnarray*}
Now from \cite{cohen4}, Prop. 9.2.12, for every $N, r \in \N$, we have 
 \begin{eqnarray*}
 \sum_{n=1}^{N}n^r &=& \frac{1}{r+1}\left(N^{r+1}+\frac{r+1}{2}N^r+\sum_{m=2}^{r}\binom{r+1}{m}B_mN^{r+1-m}\right)
 \end{eqnarray*}
 Note that  $B_0=1$, and $B_1=-\frac{1}{2}$ (See  \cite{apostol1}, section 12.12). Also, for $m\geq 1$, $B_{2m+1}=0$ (See \cite{apostol1} Theorem 12.16). Therefore, $B_3, B_5, \ldots $ are all equal to 0. So we can rewrite the above expression as 
 \begin{eqnarray*}
  \sum_{n=1}^{N}n^r &=& \frac{1}{r+1}\bigg((-1)^0 B_0 \binom{r+1}{0} N^{r+1-0}(-1)^1 B_1  \binom{r+1}{1} N^{r+1-1}\\ &+&\sum_{m=2}^{r}\binom{r+1}{m}B_m N^{r+1-m}\bigg)
  \end{eqnarray*}
  In this expression, after $B_1$, only even order Bernoulli numbers survive and so we need to consider only such terms in the summation. Rewriting using this information
   \begin{eqnarray*}
    \sum_{n=1}^{N}n^r &=& \frac{1}{r+1}(-1)^1 B_1  \binom{r+1}{1} N^{r+1-1} + \frac{1}{r+1}\sum_{m=0}^{\lfloor \frac{r}{2}\rfloor}\binom{r+1}{2m} B_{2j} N^{r+1-2m}\\
    &=& \frac{N^r}{2} + \frac{1}{r+1}\sum_{m=0}^{\lfloor \frac{r}{2}\rfloor}\binom{r+1}{2m} B_{2m} N^{r+1-2m}
    \end{eqnarray*}
  
Using this new expression, we get 
\begin{eqnarray*}
S_r^{(s)}(k) &=& \sum_{d|k} \frac{d^{s(r+1)}}{k^{s(r+1)}} \mu(\frac{k}{d}) 
\left[ \left(\frac{k^s}{d^s}\right)^r\frac{1}{2} + \frac{1}{r+1}\sum_{m=0}^{\lfloor \frac{r}{2}\rfloor}\binom{r+1}{2m} B_{2m} \left(\frac{k^s}{d^s}\right)^{r+1-2m} \right]\\
&=& \frac{1}{2}\sum_{d|k} \frac{d^s}{k^s} \mu(\frac{k}{d}) + \frac{1}{r+1} \sum_{m=0}^{\lfloor \frac{r}{2}\rfloor}\binom{r+1}{2m} B_{2m}  \sum_{d|k} \frac{d^{2ms}}{k^{2ms}} \mu(\frac{k}{d})\\
&=& \frac{1}{2} \frac{J_s(k)}{k^s} + \frac{1}{r+1} \sum_{m=0}^{\lfloor \frac{r}{2}\rfloor}\binom{r+1}{2m} B_{2m}  \frac{J_{2ms}(k)}{k^{2ms}}
\end{eqnarray*}
and this is what we claimed.
\end{proof}
When $s=1$, the above identity becomes the one given by Alkan (Prop. \ref{prop:toth-main}).
\section{$\log$ as weight}
We proceed now to generalize the other weighted averages computed by \toth (\cite{toth}).
The next one uses the logarithmic function as weight.
\begin{prop}
For every $s, k\in\N$, we have $$\frac{1}{k}\sum_{j=1}^{k} \log j \,\ckjgen = s\Lambda(k)+ \sum_{d|k} \frac{d^s}{k}\mukd\log \left(\frac{k}{d^s}\right)! $$
\end{prop}
\begin{proof}
We use identity (\ref{eq:rs-gen-mu}) to proceed:
\begin{eqnarray*}
\sum_{j=1}^{k} \log j \,\ckjgen &=& \sum_{j=1}^{k} \log j \, \sum_{\substack{d|j^s\\d|k}} d^s\mukd\\
&=& \sum_{d|k} d^s \mukd \sum_{m=1}^{\kbyds}\log(md^s)\\
&=& \sum_{d|k} d^s \mukd \sum_{m=1}^{\kbyds}[\log(1.d^s) + \log(2.d^s)+\ldots + \log(\kbyds d^s)]\\
\end{eqnarray*}
\begin{eqnarray*}
&=& \sum_{d|k} d^s \mukd \log[(\kbyds)! (d^s)^{\kbyds}]\\
&=& \sum_{d|k} d^s s \kbyds \log d \, \mukd + \sum_{d|k} d^s \mukd  \log(\kbyds)! \\
&=& sk\sum_{d|k} \mukd  \log d  + \sum_{d|k} d^s \mukd  \log(\kbyds)! 
\end{eqnarray*}
Note that $\sum\limits_{d|k} \mukd  \log d = (\mu * \log)(k) = \Lambda(k)$. So we get
\begin{eqnarray*}
\frac{1}{k}\sum_{j=1}^{k} \log j \,\ckjgen &=& s\Lambda(k)  + \sum_{d|k} \frac{d^s}{k} \mukd  \log(\kbyds)! 
\end{eqnarray*}
\end{proof}
\section{Function of generalized GCD as weight}
Now we proceed to establish an identity with weight as the generalized GCD function. But instead, we prove it generally for arithmetic functions composed with the GCD function. This, on taking the first function to be the function $N$ as a special case, we get the weight as the generalized GCD function.
\begin{prop}
For $s, k, j\in\N$  $$\sum_{j=1}^{k^s}f((j, k^s)_s)\ckjgen = J_s(k)[(f\circ N^s) * (\mu \circ N^s)] (k)$$
\end{prop}
\begin{proof}
We use the generalized Holder's evaluation (\ref{ghi}) to start off.
\begin{eqnarray*}
\sum_{j=1}^{k^s}f((j, k^s)_s)\ckjgen = \sum_{j=1}^{k^s}f((j, k^s)_s) \ghi
\end{eqnarray*}
Now we group the terms in the sum according to $d^s= (j, k^s)_s$,

\begin{eqnarray*}
\sum_{j=1}^{k^s}f((j, k^s)_s)\ckjgen &=& J_s(k) \sum_{d^s|k^s} f(d^s) \frac{\muksds}{\jsksds} \sum_{\substack{m=1\\(m, \ksds)_s=1}}^{\ksds} 1\\
&=& J_s(k) \sum_{d^s|k^s} f(d^s) \frac{\muksds}{\jsksds} \jsksds\\
&=& J_s(k) \sum_{d^s|k^s} f(d^s) \muksds
\end{eqnarray*}
which is equivalent to the identity we have to prove.
\end{proof}

\section{The Gamma function as weight}
%Recall that $$\frac{J_s(k)}{k^s} = \prod\limits_{\substack{p|k\\p\, prime}}\left(1-\frac{1}{p^s}\right)$$. 
To establish the next average, we need a few preliminary lemmas. We have the following identity which appears in \cite{cohen4} as exercise 1.8.45:
\begin{lemma}
For $s, k\in\N$
$$\sum_{d|k}\frac{\mu(d)}{d} \log d = -\frac{\varphi(k)}{k}\sum_{\substack{p|k \\p\, prime}} \frac{\log p}{p-1}$$
\end{lemma}

We will give a generalization of this lemma in terms of Jordan totient function. This lemma could be of some other independent interest.

\begin{lemma}\label{prop:mu-log-d}
For $s, k\in\N$
$$\sum_{d|k}\frac{\mu(d)}{d^s} \log d = -\frac{J_s(k)}{k^s}\sum_{\substack{p|k \\p\, prime}} \frac{\log p}{p^s-1}$$
\end{lemma}

\begin{proof}
When we take all the divisors of $k$, only those $\mu(d)$ survive in the computation where $d$ is square free. Now recall that $$\frac{J_s(k)}{k^s} = \prod\limits_{\substack{p|k\\p\, prime}}\left(1-\frac{1}{p^s}\right).$$ So the quotient $\frac{J_s(k)}{k^s}$ considers each prime factor in a divisor $d$ of $k$ only once. Therefore, without any loss of generality we may assume that $k$ is square free. Let $k = \prod\limits_{p_i|k} p_i$ where $p_i$ are distinct primes and $D = \{p_1, \ldots, p_r\}$ denote the set of all prime divisors of $k$.
With this assumption and notation, the the RHS of the expression we want to verify becomes
\begin{eqnarray*}
\frac{J_s(k)}{k^s}\sum_{\substack{p_i\in D}} \frac{\log p_i}{p_i^s-1} &=& \prod_{p_i\in D}(1-\frac{1}{p_i^s})\sum_{p_i\in D} \frac{\log p_i}{p_i^s - 1}\\
&=&  \prod_{p_i\in D}\frac{p_i^s - 1}{p_i^s} \frac{\sum\limits_{p_i\in D} \left(\prod\limits_{\substack{p_j\in D\\p_j\neq p_i}}(p_j^s - 1) \log p_i \right)}{ \prod\limits_{p_i\in D} (p_i^s - 1)}\\
&=&   \frac{\sum\limits_{p_i\in D} \left(\prod\limits_{\substack{p_j\in D\\p_j\neq p_i}}(p_j^s - 1) \log p_i \right)}{ \prod\limits_{p_i\in D} p_i^s}
\end{eqnarray*}
We introduce a notation for convenience. $S_{n, p_i}$ will stand for the set of $n$ primes (where $n<r$) from $D$ and contains $p_i$ and $S_n$ will stand for the set of all $S_{n, p_i}$ where $p_i\in D$. Therefore, $S_{r, p_i} = D $. With this notation, we get
\begin{eqnarray*}
\frac{J_s(k)}{k^s}\sum_{\substack{p_i\in D}} \frac{\log p_i}{p_i^s-1} &=&  \left(\frac{1}{\prod\limits_{p_i\in D} p_i^s}\right) \left[\sum\limits_{p_i\in D} \log p_i \left(\prod\limits_{\substack{p_j\in D\\p_j\neq p_i}}(p_j^s - 1)  \right)\right]\\
&=&  \left(\frac{1}{\prod\limits_{p_i\in D} p_i^s}\right) \left[\sum\limits_{p_i\in D} \log p_i \left(\prod\limits_{\substack{p_j\in D\\p_j\neq p_i}}p_j^s - \sum\limits_{S_{r-1,p_i}\in S_{r-1}} \prod\limits_{\substack{p_j\in S_{r-1}\\p_j\neq p_i}}p_j^s + \ldots +\right)\right]\\
&=&   \sum\limits_{p_i\in D} \log p_i \left(\frac{1}{p_i^s} - \sum\limits_{S_{2,p_i}\in S_{2}}\frac{1}{\prod\limits_{\substack{p_j\in S_{2}}}p_j^s} + \ldots +\right)
\end{eqnarray*}

Computing the LHS now:
\begin{eqnarray*}
\sum_{d|k}\frac{\mu(d)}{d^s} \log d &=& -\frac{\log p_1}{p_1^s}\\
&+& -\left(\frac{\log p_1p_2}{p_1^s p_2^s} +\ldots + \frac{\log p_{r-1}p_r}{p_{r-1}^s p_r^s}\right)\\
&+& \frac{\log p_1p_2p_3}{p_1^s p_2^sp_3^s} + \ldots + \frac{\log p_{r-2}p_{r-1}p_r}{p_{r-2}^s p_{r-1}^s p_r^s}\\
&+& \ldots + (-1)^r \frac{\log p_1 \ldots p_r}{p_1^s \ldots p_r^s}
\end{eqnarray*}
grouping the terms on  $p_i$,
\begin{eqnarray*}
\sum_{d|k}\frac{\mu(d)}{d^s} \log d &=&  \sum\limits_{p_i\in D} \log p_i\bigg(-\frac{1}{p_i^s} + \sum_{S_{2, p_i}\in S_2}\frac{1}{\prod\limits_{p_j\in S_{2, p_i}}p_j^s} -  \sum_{S_{3, p_i}\in S_3}\frac{1}{\prod\limits_{p_j\in S_{3, p_i}}p_j^s}\\
&+& \ldots \\
&+& (-1)^k \sum_{S_{k, p_i}\in S_k}\frac{1}{\prod\limits_{p_j\in S_{k, p_i}}p_j^s}\\
&+& \ldots + (-1)^{r-1} \sum_{S_{r-1, p_i}\in S_{r-1}}\frac{1}{\prod\limits_{p_j\in S_{r-1, p_i}}p_j^s}\\
&+& (-1)^r \frac{1}{\prod\limits_{p_j\in D}p_j^s}\bigg)\\
\end{eqnarray*}
which is the same as the RHS computed above.
\end{proof}

The above lemma will be used in the proof of the next proposition.
\begin{prop}
For $s,k\in\N$
$$\frac{1}{J_s(k)}  \sum_{j = 1}^{k^s} \log\Gamma(j/k^s)\ckjgen = \frac{s}{2}\sum_{p|k} \frac{\log p}{p^s - 1} - \frac{\log 2\pi}{2}$$
\end{prop}
\begin{proof}
\begin{eqnarray*}
 \sum_{j = 1}^{k^s} \log\Gamma(j/k^s)\ckjgen &=& \sum_{j = 1}^{k^s} \log\Gamma(j/k^s)\sum_{\substack{d^s|j,\, d^s|k^s}}d^s\mukd\\
&=&  \sum_{\substack{ d^s|k^s}}d^s\mukd\sum_{m = 1}^{\ksds} \log\Gamma(md^s/k^s)\\
&=&  \sum_{\substack{ d^s|k^s}}d^s\mukd \log\prod_{m = 1}^{\ksds}\Gamma(\frac{m}{\ksds})
\end{eqnarray*}
Now \cite{cohen4}, Proposition 9.6.33 tells that $$\prod_{1\leq j\leq N}\Gamma(j/N) = \frac{(2\pi)^{(n-1)/2}}{\sqrt{n}}$$
Using this, we get 

\begin{eqnarray*}
 \sum_{j = 1}^{k^s} \log\Gamma(j/k^s)\ckjgen &=&  \sum_{\substack{ d^s|k^s}}d^s\mukd \log \frac{(2\pi)^{(\ksds-1)/2}}{\sqrt{\ksds}}\\
  &=&  \sum_{\substack{ d^s|k^s}}d^s\mukd\left(\frac{1}{2}(\ksds-1)\log 2\pi - \frac{1}{2}\log \ksds\right)\\
  &=& \frac{k^s}{2} \log 2\pi  \sum_{\substack{ d|k}}\mukd - \frac{1}{2}\log 2\pi \sum_{\substack{ d|k}}d^s\mukd  - \frac{1}{2} \sum_{\substack{ d|k}}d^s\mukd \log (\ksds)
 \end{eqnarray*}
 It is a standard fact that (see \cite{apostol1})  $\sum\limits_{ d|k}\mukd = 0$. Also, taking $k/d$ instead of $d$ as the divisors of $k$, we get

 \begin{eqnarray*}
  \sum_{j = 1}^{k^s} \log\Gamma(j/k^s)\ckjgen &=&  -\frac{1}{2}J_s(k)\log 2\pi - \frac{1}{2} \sum_{\substack{ d|k}}\ksds\mu(d) \log (d^s)\\
&=&  -\frac{1}{2}J_s(k)\log 2\pi - \frac{sk^s}{2} \times -\frac{J_s(k)}{k^s}\sum_{\substack{p|k \\p\, prime}} \frac{\log p}{p^s-1}\\
& &\text{ using Lemma (\ref{prop:mu-log-d})}\\
   \end{eqnarray*}
 \begin{eqnarray*}
&=&  J_s(k)\left(\frac{s}{2} \sum_{\substack{p|k \\p\, prime}} \frac{\log p}{p^s-1} - \frac{1}{2}\log 2\pi \right)\\
  \end{eqnarray*}
  and this completes the proof.
\end{proof}
\section{Bernoulli numbers as weight}
We now consider the weighted average involving Bernoulli numbers. Even though a similar identity  was stated in \cite{toth}, no proof was given there. Our next proposition gives a generalization of this proposition, and a special case to of turns out to be  \toth's identity given in \cite{toth}.

\begin{prop}
For $m\in\Z, m\geq 0$
$$\frac{1}{k^s}\sum_{j=0}^{k^s-1} B_m(\frac{j}{k^s})\ckjgen = \frac{B_m}{k^{sm}}J_{sm}(k)$$
\end{prop}
\begin{proof}
\begin{eqnarray*}
\sum_{j=0}^{k^s-1} B_m(\frac{j}{k^s})\ckjgen &=& \sum_{j=0}^{k^s-1} B_m(\frac{j}{k^s})\, d^s\mukd\\
 &=& \sum_{d|k} d^s\mukd \sum_{\substack{j=0\\d^s|j}}^{k-1} B_m(\frac{j}{k})
\end{eqnarray*}
Now use the fact that $B_m(0) = B_m(1)$.
\begin{eqnarray*}
\sum_{j=0}^{k^s-1} B_m(\frac{j}{k^s})\ckjgen &=&  \sum_{d|k} d^s\mukd \left(\sum_{\substack{j=1\\d^s|j}}^{k^s} B_m(\frac{j}{k^s}) + B_m(\frac{0}{k^s}) - B_m(\frac{k^s}{k^s})\right)\\
&=&  \sum_{d|k} d^s\mukd \sum_{\substack{j=1\\d^s|j}}^{k^s} B_m(\frac{j}{k^s})\\
\end{eqnarray*}
\begin{eqnarray*}
&=&  \sum_{d|k} d^s\mukd \sum_{r=1}^{\ksds} B_m(\frac{rd^s}{k^s})\\
&=&  \sum_{d|k} d^s\mukd \sum_{r=1}^{\ksds} B_m(\frac{r}{\ksds})\\
&=&  \sum_{d|k} d^s\mukd \sum_{r=0}^{\ksds - 1} B_m(\frac{r}{\ksds}) - B_m(\frac{0}{\ksds}) + B_m(\frac{\ksds}{\ksds})\\
\end{eqnarray*}
From the properties  Bernoulli numbers  (\cite{cohen4}, Prop. 9.1.3), we have $$\sum_{j=0}^{k-1}B_m(\frac{j}{k}) = \frac{B_m}{k^{m-1}}$$
Using this, we get
\begin{eqnarray*}
\sum_{j=0}^{k^s-1} B_m(\frac{j}{k^s})\ckjgen &=& \sum_{d|k} d^s\mukd  \frac{B_m}{(\ksds)^{m-1}}\\
&=& \frac{B_m}{(k^s)^{m-1}} \sum_{d|k} d^{sm}\mukd  \\
&=& \frac{B_m}{(k^s)^{m-1}} J_{sm}(k) 
\end{eqnarray*}
Now the conclusion easily follows.

\end{proof}

\section{Binomial coefficients as weight}
Now we have one more weighted average with weights as binomial coefficients.
\begin{prop}
For $k\in\N$,
$$\sum_{j=0}^{k^s} \binom{k^s}{j} \ckjgen = 2^{k^s} \sum_{d|k}\mukd\sum_{l=1}^{d^s}(-1)^{l\ksds}\cos^{k^s} ( l\pi/d^s)$$
\end{prop}

\begin{proof}
\begin{eqnarray*}
\sum_{j=0}^{k^s} \binom{k^s}{j} \ckjgen &=& \sum_{j=0}^{k^s} \binom{k^s}{j} \sum_{d|k,\, d^s|l} d^s\mukd\\
&=& \sum_{d|k} d^s\mukd\sum_{m=0}^{\ksds} \binom{k^s}{md^s} \\
&=& \sum_{d|k} d^s\mukd\sum_{m=0}^{\ksds} \binom{k^s}{md^s} \\
\end{eqnarray*}
Now we use the following formula appearing in \cite{comtet}, P. 84:
$$\sum_{m=0}^{\lfloor n/r\rfloor} \binom{n}{mr} = \frac{2^n}{r}\sum_{l=1}^{r}\cos^n(l\pi/r) \cos(nl\pi/r)$$
to get
\begin{eqnarray*}
\sum_{j=0}^{k^s} \binom{k^s}{j} \ckjgen &=& \sum_{d|k} d^s\mukd \frac{2^{k^s}}{d^s} \sum_{l=1}^{d^s}\cos^{k^s}(l\pi/d^s)\cos(k^sl\pi/d^s)
\end{eqnarray*}
Note that $k^s/d^s$ is an integer and so $\cos(k^sl\pi/d^s)$ is $+1$ if $\ksds l$ is even and $-1$ if it is odd. So we get the final expression which was claimed.
\end{proof}
\section{Exponential function as weight}
Though the following proposition  was stated in \cite{toth} as a remark, detailed proof was not given there. We  provide a proof here before generalizing the proposition.
\begin{prop}\label{prop:exp-toth}
For every $k, n \in \N$, 
$$\frac{1}{k}\sum_{j=1}^{k} exp(2\pi ij n/k)\ckj = \begin{cases}
1, \text {if } (k,n)=1\\
0, \text {if } (k,n) >1
\end{cases}$$
\end{prop}
\begin{proof}
We use the Fourier expansion of the functions $\theta_k(n)$ (defined in the section \emph{Notations}) and $c_k(n)$ to prove the identity. We have
\begin{eqnarray*}
\frac{1}{k}\sum_{j=1}^{k}\theta_k(j) exp(-2\pi i j n/k) &=& \frac{1}{k}\sum_{\substack{j=1\\(j,k)=1}}^{k} exp(-2\pi i j n/k) = \frac{c_k(n)}{k}. 
\end{eqnarray*}
Threfore Using the Fourier expansion identities, we get
$$\sum_{j=1}^{k}\frac{c_k(j)}{k} exp(2\pi i j n/k) = \theta_k(n)$$.
\end{proof}
The generalization follows easily:
\begin{prop}
For $k, n\in\N$, we have
$$\frac{1}{k^s} exp(2\pi i j n/k^s) \ckjgen=\begin{cases}
1 \text{ if } (k, n^s)_s = 1\\
0 \text{ if } (k, n^s)_s > 1
\end{cases} =:\theta^{(s)}_k(n)$$
\end{prop}
\begin{proof}
\begin{eqnarray*}
\frac{1}{k^s}\sum_{j=1}^{k^s}\theta^{(s)}_k(j) exp(-2\pi i j n/k^s) &=& \frac{1}{k^s}\sum_{\substack{j=1\\(j,k^s)_s=1}}^{k^s} exp(-2\pi i j n/k^s) = \frac{c^{(s)}_k(n)}{k^s}. 
\end{eqnarray*}
Threfore Using the Fourier expansion identities, we get
$$\sum_{j=1}^{k^s}\frac{c^{(s)}_k(j)}{k^s} exp(2\pi i j n/k^s) = \theta^{(s)}_k(n)$$ which is what we require.
\end{proof}

\section{Generalization of a multivariate identity}
Finally we have a generalization of a multivariate identity given in \cite{toth}:

\begin{prop}
 Let $k_1, \ldots, k_n\in\N$, $k=lcm(k_1,\ldots,k_n)$ and let $s,r\in\N$. Then
 \begin{equation}
  S_r^{(s)}(k_1, \ldots, k_n)=\frac{J_s(k_1)\ldots J_s(k_n)}{2k^s}+\frac{1}{r+1}\sum_{m=0}^{\lfloor r/2\rfloor}  \binom{r+1}{2m}\frac{B_{2m}}{k^{2ms}}g_m(k_1,\ldots,k_n)
 \end{equation}
where
\begin{equation}
 g^{(s)}_m(k_1,\ldots k_n)=\sum_{d_1|k_1, \ldots, d_n|k_n}\frac{d_1^s\mu(k_1/d_1)\ldots d_n^s\mu(k_n/d_n)}{(lcm(d_1,\ldots,d_n))^{(1-2m)s}}
\end{equation}
is a multiplicative function in $n$ variables.
\end{prop}
This can be easily verified following the steps given in the proof of the single variable case.
With $r=1$, we have

\begin{corollary} 
With notations as in the above proposition, if we let $E^{(s)}: = g_1$, then
$$ S^{(s)}_1(k_1, \ldots, k_n)=\frac{J_s(k_1)\ldots J_s(k_n)}{2k} + \frac{E^{(s)}(k_1,\ldots, k_n)}{2}$$
\end{corollary}
\section{Further directions}
Though we have found several identities for various arithmetical functions combined with the generalized Ramanujan sums, we restricted ourselves to the finite summation case. If one looks into what happens when we extend this to an infinite series, it is possible that it may give us many more ``Ramanujan expansions'' of many well known arithmetical functions in terms of the generalized sums.
Also, it should be possible to derive similar identities based on the other generalizations of Ramanujan sums with the same techniques.

\end{document}